\documentclass[11pt]{article}
\usepackage{verbatim} 
\usepackage{amsmath}
\usepackage{amssymb}
\usepackage{amsfonts}
\usepackage{amsthm}
\usepackage{cite}
\usepackage{graphicx} 
\usepackage{pgf,tikz}
\usepackage{mathrsfs}
\usetikzlibrary{arrows}
\usepackage{color}
\usepackage{enumitem}
\usepackage{latexsym}
\usepackage{enumerate}

\DeclareMathOperator{\ord}{ord}
\DeclareMathOperator{\bd}{bd}

\newtheorem{theorem}{Theorem}

\newtheorem{lemma}[theorem]{Lemma}
\newtheorem{proposition}[theorem]{Proposition}
\newtheorem{fact}[theorem]{Fact}
\theoremstyle{definition}
\newtheorem{note}[theorem]{Note}
\newtheorem{definition}[theorem]{Definition}
\newtheorem{question}[theorem]{Question}
\newtheorem{claim}[theorem]{Claim}

\newcommand{\R}{\mathbb{R}}
\newcommand{\N}{\mathbb{N}}

\begin{document}          

\makeatletter
\@namedef{subjclassname@2010}{
  \textup{2010} Mathematics Subject Classification 03E15, 54H05 and 54F15}
\makeatother

\title{The complexity of homeomorphism relations on some classes of compacta}

\author{
Pawe\l\ Krupski\footnote{Partially supported by the Faculty of Fundamental Problems of Technology, Wroc\l aw University of Science and Technology, grant 0401/0017/17}\\
Department of Computer Science, Faculty of Fundamental Problems of Technology \\ Wroc\l aw University of Science and Technology, Wroc\l aw, Poland \\
\\
Benjamin Vejnar\footnote{The author was supported by the grant GA\v CR 17-04197Y.
This work has been supported by Charles University Research Centre program No.UNCE/SCI/022.
}
\\ Department of Mathematical Analysis
\\ Faculty of Mathematics and Physics, Charles University, Prague, Czechia\\
}

\date{\today}
\maketitle

\centerline{\sc Dedicated to the memory of V\v{e}ra Trnkov\'a.}

\begin{abstract}

We prove that the homeomorphism relation between compact spaces can be continuously reduced to the homeomorphism equivalence relation between absolute retracts which strengthens and simplifies recent results of Chang and Gao, and Cie\'sla. It follows then that the homeomorphism relation of absolute retracts is Borel bireducible with the universal orbit equivalence relation.
We also prove that the homeomorphism relation between regular continua is classifiable by countable structures and hence it is Borel bireducible with the universal orbit equivalence relation of the permutation group on a countable set. On the other hand we prove that the homeomorphism relation between rim-finite metrizable compacta is not classifiable by countable structures.
\end{abstract}

Keywords: Borel reduction,  countable structures, universal orbit equivalence relation, absolute retract, $LC^n$-compactum, regular continuum.

Mathematics Subject Classification 03E15, 54H05 and 54F15.

\section{Introduction}
 Measuring the complexity of structures is a very general task. Usually we study the complexity in a relative way by comparing with other structures.
One possible approach for this by using embeddings of categories was elaborated by Pultr and Trnkov\'a in \cite{PultrTrnkova}.
However, in  this paper we use the notions of \emph{Borel reducibility} and \emph{Borel bireducibility} to relate the complexities of equivalence relations on standard Borel spaces or Polish spaces in order to compare the complexities of classification problems. For more details we refer to the book of Gao \cite{Gao}.

 Several equivalence relations  became milestones in the theory of Borel reductions, we will mention four of them:
\begin{itemize}[noitemsep]
\item the equality on an uncountable Polish space,
\item the $S_\infty$-universal orbit equivalence relation, 
\item the universal orbit equivalence relation, 
\item the universal analytic equivalence relation.
\end{itemize}

Let us give several \emph{examples} to make the above relations more familiar.
A classical example is a result of Gromov (see e.g. \cite[Theorem~14.2.1]{Gao}) who proved that the isometry equivalence relation of compact metric spaces is a \emph{smooth} equivalence relation, which means that it is Borel reducible to the equality of real numbers (or equivalently of an uncountable Polish space).
The isomorphism relation of countable graphs or the isomorphism relation of countable linear orders are Borel bireducible with the $S_\infty$-universal orbit equivalence.
The isometry equivalence relation of Polish metric spaces and the isometry relation of separable Banach spaces were proved by Gao and Kechris in \cite{GaoKechris} and by Melleray in \cite{Melleray} to be Borel bireducible with the universal orbit equivalence relation (see a survey paper by Motto Ros~\cite{MottoRos}). 
Ferenczi, Louveau and Rosendal proved in \cite{FerencziLouveauRosendal} that the isomorphism equivalence relation of separable Banach spaces is Borel bireducible with the universal analytic equivalence relation.

In order to capture all the structures in one space we need some sort of \emph{coding}. This can be done by considering some universal space (e.g. the Hilbert cube, the Urysohn space or the Gurari\u\i\ space) and all its subspaces with some natural Polish topology or Borel structure (e.g the hyperspace topology or the Effros Borel structure).
Sometimes there are other natural ways to encode a given structure. For example the class of Polish metric spaces can be coded by the set of all metrics on $\N$ where two metrics are defined to be equivalent if the completions of the respective spaces are isometric. Fortunately by \cite[Theorem 14.1.3]{Gao} it does not matter which coding we choose.
It is generally believed that this independence on a natural coding is common to other structures and thus the statements are usually formulated for all structures without mentioning the current coding.

In this paper we are dealing solely with the classification problem of compact metrizable spaces up to homeomorphism. Unless otherwise stated we assume that the coding of compact metrizable spaces is given by the hyperspace of the Hilbert cube. Zielinski proved in \cite{Zielinski} that the homeomorphism equivalence relation of metrizable compacta is Borel bireducible with the universal orbit equivalence relation. Soon after Chang and Gao proved in \cite{ChangGao} that a similar result is true if we restrict the relation only to continua. Recently Cie\'sla proved in \cite{Ciesla} that a similar result holds true also for locally connected continua.
We prove in Theorem~\ref{compAR} that the homeomorphism relation can be restricted to absolute retracts still having the same complexity.
By a similar idea we prove in Theorem~\ref{ndim} that at most $n$-dimensional continua are Borel reducible to at most $n$-dimensional $LC^{n-1}$-continua, for every $n\in\N$. However, the exact complexity of at most $n$-dimensional compacta remains unknown.

Camerlo, Darji and Marcone proved in \cite{CamerloDarjiMarcone}
that the homeomorphism relation of dendrites is Borel bireducible with the $S_\infty$-universal orbit equivalence relation, especially dendrites are classifiable by countable structures.
A similar result was proved by Camerlo and Gao in \cite{CamerloGao} for zero-dimensional compacta instead of dendrites.
We prove in Theorem~\ref{rimfinitecontinua}, that the homeomorphism equivalence relation of regular continua has the complexity of countable structures. On the other hand it is surprising that the homeomorphism relation of rim-finite compacta, which is a natural small class containing all dendrites and zero-dimensional compacta, is not classifiable by countable structures; see Theorem~\ref{rimfinitecompacta}.

\section{Definitions and notations}

Recall that a  \emph{standard Borel space} is a measurable space $(X, \mathcal S)$ such that there is a Polish topology $\tau$ on $X$ for which the family $\text{Borel}(X, \tau)$ of Borel subsets of $(X,\tau)$ is equal to $\mathcal S$. In order to compare the complexities of equivalence relations we use the notion of Borel reducibility.

\begin{definition}
Suppose that $X$ and $Y$ are standard Borel spaces and let $E$, $F$ be (usually analytic) equivalence relations on $X$ and $Y$ respectively. We say that $E$ is \emph{Borel reducible} to $F$, and we denote this by $E\leq_B F$, if there exists a Borel mapping $f\colon X\to Y$ such that
\[x E x' \iff f(x) F f(x'),\]
for every $x, x'\in X$.
The function $f$ is called a \emph{Borel reduction}.
We say that $E$ is \emph{Borel bireducible} with $F$, and we write $E \sim_B F$, if $E$ is Borel reducible to $F$ and $F$ is Borel reducible to $E$.
\end{definition} 
In a similar fashion we can define being \emph{continuously reducible} if we suppose that $X$ and $Y$ are Polish spaces and $f$ is continuous.

In the whole paper we denote $I=[0,1]$ and $Q=I^\N$. 
For a separable metric space $X$ we denote by $2^X$ the hyperspace of all compacta in $X$ a with the Hausdorff distance $d_H$. The hyperspace $C(X)$ stands for all continua in $X$.
Since we are dealing mainly with the \emph{homeomorphism equivalence relation} between elements of some  class  $\mathcal C$ of metric compacta, we rectrict the relation to $\mathcal C\cap 2^{Q}$. For example, we say that the class of continua is continuously reducible to the class of compacta instead of the longer expression that the homeomorphism relation of the class of all continua in the Hilbert cube is continuously reducible to the homeomorphism relation of all compacta in the Hilbert cube.

We say, that an equivalence relation $E$ defined on a standard Borel space $X$ is \emph{classifiable by countable structures} if there is a countable relation language $\mathcal L$ such that $E$ is Borel reducible to the isomorphism relation of $\mathcal L$-structures whose underlying set is $\N$.
An equivalence relation $E$ on a standard Borel space $X$ is said to be an \emph{orbit equivalence relation} if there is a Borel action of a Polish group $G$ on $X$ such that $xEx'$ if and only if  there is some $g\in G$ for which $gx=x'$.

Let $\mathcal C$ be a class of equivalence relations on standard Borel spaces. An element $E\in\mathcal C$ is called \emph{universal} for $\mathcal C$ if $F\leq_B E$ for every $F\in\mathcal C$.
It is known that for every Polish group $G$ there is an equivalence relation (denoted by $E_G$) on a standard Borel space that is universal for all orbit equivalence relations given by $G$-actions.
We are particularly interested in the \emph{universal $S_\infty$-equivalence relation} $E_{S_\infty}$.
It is known that an equivalence relation is classifiable by countable structures if and only if it is Borel reducible to $E_{S_\infty}$.
Also there exists \emph{universal orbit equivalence relation} which is denoted by $E_{G_\infty}$. We should also note that all the mentioned equivalence relations are analytic.

An \emph{absolute retract} is any retract of the Hilbert cube.
A continuum is called \emph{arc-like} if it is an inverse limit of arcs with continuous and onto bonding mappings. A \emph{dendrite} is a locally connected continuum which does not contain a simple closed curve. It worth mentioning that a nondegenerate dendrite is fully characterized as an absolute retract of dimension one.
A compact space is said to be \emph{rim-finite} if it has an open base formed by sets with finite boundaries. A rim-finite continuum is usually called a \emph{regular continuum}.  Recall here that if $X$ is a regular continuum then $X$ is hereditarily locally connected, thus a closed set $A\subseteq X$ separates $X$ between two distinct points $x,y\in X\setminus A$ if and only if $A$ cuts $X$ between $x$ and $y$ if and only if every arc from $x$ to $y$ intersects $A$. 

In the next section we are going to use properties  of $Z$-sets in the Hilbert cube or in the Menger universal continuum $M_n$ (~\cite{vanMill}  and~\cite{Bestvina} are good references). The crucial facts are: 
\begin{itemize}
\item
the $Z$-set unknotting theorem: each homeomorphism between $Z$-sets in $Q$ or in $M_n$ extends to a homeomorphism of $Q$ or $M_n$, respectively, onto itself~\cite[Theorem 5.3.7]{vanMill} and~\cite[Corollary 3.1.5]{Bestvina},
\item
$Z$-approximation theorem: if $B$ is a compactum such that $\dim B\le n$, then any continuous map $f:B\to M_n$ can be approximated by a $Z$-embedding (i.e., by an embedding such that $f(B)$ is a $Z$-set in $M_n$)~\cite[Theorem 2.3.8]{Bestvina},
\item
each closed subset of a $Z$-set is a $Z$-set,
\item
the Hilbert cube as well as the Menger universal continua contain their respective copies as $Z$-sets. In fact $Q\times\{0\}$ is a $Z$-set in $Q\times I$.

\end{itemize}

%
For the definition and basic properties of the topological dimension and continuum theory we refer the reader to \cite[Chapter 7]{Engelking} and to \cite{Nadler},~\cite{vanMill} and~\cite{Whyburn}.

\section{Main results}

We present three classification results in theorems \ref{compAR}, \ref{ndim} and \ref{rimfinitecontinua}. Also there is one non-classification result in Theorem~\ref{rimfinitecompacta}.
The following proposition is a special case of \cite[Proposition~1]{Zielinski} and it can be proved using the back and forth argument.

\begin{proposition}\label{extension}
Let $K\subseteq A, L\subseteq B$ be four nonempty compact metrizable spaces such that $A\setminus K$ and $B\setminus L$ are dense sets of isolated points in $A$ and $B$ respectively. Then every homeomorphism of $K$ onto $L$ can be extended to a homeomorphism of $A$ onto $B$.
\end{proposition}

The following lemma is a consequence of~\cite{Curtis}; see~\cite[Theorem 2.6]{GvM} for more details. It will be extremely useful in the proof of Theorem~\ref{compAR}.

\begin{lemma}\label{homotopy}
If $X$ is a nondegenerate Peano continuum then there exists a homotopy $H\colon  2^X\times I\to 2^X$ for which
\begin{itemize}[noitemsep]
\item $H(A, 0)=A$ for every $A\in 2^X$,
\item $H(A, t)$ is finite for every $t>0$ and $A\in 2^X$.
\end{itemize}
\end{lemma}

Recall that if $Y\subseteq X$ and $\varepsilon>0$ we say that $X$ is $\varepsilon$-deformable into $Y$ if there exists a continuous mapping $\varphi\colon X\times [0,1]\to X$ such that $\varphi(x,0)=x$, $\varphi(x,1)\in Y$ and the diameter of $\varphi(\{x\}\times[0,1])$ is at most $\varepsilon$ for every $x\in X$.
The following proposition was proved in \cite[1.1 and 1.3]{Krasinkiewicz}.

\begin{proposition}\label{AR}
Let $X$ be a compact space such that for every $\varepsilon>0$ there exists an absolute neighborhood retract (absolute retract) $Y\subseteq X$ for which $X$ is $\varepsilon$-deformable into $Y$. Then $X$ is an absolute neighborhood retract (absolute retract, resp.).
\end{proposition}

The proof of Proposition~\ref{AR} is based on the Lefschetz characterization of absolute neighborhood retracts via extending partial realizations of finite polyhedra to their full realizations. By restricting polyhedra of arbitrary dimensions to polyhedra of dimension $\le n$, one gets a similar characterization of $LC^{n-1}$-compacta (see~\cite[Proposition 4.2.29]{vanMill}) and the proof works for $LC^{n-1}$ compacta. So we get the following  analogue of Proposition~\ref{AR}.

\begin{proposition}\label{LC}
Let $X$ be a compact space such that for every $\varepsilon>0$ there exists an $LC^{n-1}$-compactum $Y\subseteq X$ for which $X$ is $\varepsilon$-deformable into $Y$. Then $X$ is an $LC^{n-1}$-compactum.
\end{proposition}

\begin{theorem}\label{compAR}
Compacta are continuously bireducible with absolute retracts.
\end{theorem}

\begin{proof}
Let us denote
\begin{align*}
Q_0 &=Q\times\{(0,0,0)\}, \\
Q_1 &=Q\times I\times\{(0,0)\}, \\
Q_2 &=Q\times I\times I\times\{0\}, \\
Q_3 &=Q\times I\times I\times I.
\end{align*}
It follows that $Q_i$ are Hilbert cubes such that $Q_{i}$ is a $Z$-set in $Q_{i+1}$ for $i=0,1,2$. We construct a reduction $\varphi\colon 2^Q\to 2^{Q_3}$.
Suppose that $H\colon 2^Q\times I\to 2^Q$ is given by Lemma~\ref{homotopy}
and let \[\psi(K)=\left(K\times \{(0,0,0)\}\right)\cup\bigcup_{n=0}^\infty H(K, 2^{-n})\times\{(2^{-n},0,0)\}.\]
Hence $\psi(K)$ is a compact subspace of $Q_1$ containing a homeomorphic copy of $K$ and a countable set of isolated points which is open and dense in $\psi(K)$. Denote
\[\varphi(K)=Q_2\cup \bigcup_{n=0}^\infty H(K, 2^{-n})\times\{(2^{-n}, 0)\}\times [0,2^{-n}].\]
Roughly speaking, $\varphi(K)$ consists of the Hilbert cube $Q_2$ to which a null sequence of mutually disjoint segments is attached. Therefore, by Proposition~\ref{AR}, $\varphi(K)$ is an absolute retract. Continuity of the mapping $\varphi$ is straightforward so it only remains to prove that $\varphi$ is a reduction.

Suppose first that $K, L\in 2^Q$ and $h\colon K\to L$ is a homeomorphism. By Proposition~\ref{extension} there is a homeomorphism $h_1\colon \psi(K)\to \psi(L)$ extending the homeomorphism $h\times\{(0,0,0)\}$.
Since $Q_1$ is a $Z$-set in $Q_2$, there is a homeomorphism $h_2\colon Q_2\to Q_2$ which extends $h_1$.
Now, extend $h_2$ to $h_3\colon \varphi(K)\to\varphi(L)$ on the remaining arcs linearly.
Hence $\varphi(K)$ is homeomorphic to $\varphi(L)$.

On the other hand suppose that $\varphi(K)$ is homeomorphic to $\varphi(L)$ by some homeomorphism $g$.
Since $g$ as well as $g^{-1}$ needs to map points of order two (local cut points) to points of order two (local cut points, resp.) it follows that $g$ is sending the following countable set
\[\bigcup_{n=0}^\infty H(K, 2^{-n})\times\{(2^{-n}, 0,0)\}\]
onto the set
\[\bigcup_{n=0}^\infty H(L, 2^{-n})\times\{(2^{-n}, 0,0)\},\]
because these sets are formed exactly by local cut points of $\varphi(K)$ and $\varphi(L)$, respectively,  which are not of order two.
Since $K\times\{(0,0,0)\}$ forms the set of cluster points of the first set and $L\times\{(0,0,0)\}$ forms the set of cluster points of the second one, it means that  $g(K\times\{(0,0,0)\})=L\times\{(0,0,0)\}$. Hence $K$ and $L$ are homeomorphic.
\end{proof} 

An $n$-dimensional  analogue of the Hilbert cube is the universal $n$-dimensional  Menger continuum $M_n$ which is characterized as an $n$-dimensional $LC^{n-1}$, $C^{n-1}$ continuum that satisfies the disjoint $n$-cells property~\cite{Bestvina}. Also,  $LC^{n-1}$ compacta are natural analogues of compact absolute neighborhood retracts  among  compacta of dimension $\leq n$ and $LC^{n-1}$, $C^{n-1}$-continua of such dimensions correspond to absolute retracts. Therefore a question arises whether the class $\mathcal C_n$ of all compacta of dimension $\leq n$  is continuously bireducible with the class $\mathcal C_n\cap LC^{n-1}\cap C^{n-1}$ of at most $n$-dimensional $LC^{n-1}$, $C^{n-1}$ continua. Notice that $\mathcal C_n$ is a $G_{\delta}$-subset of $2^Q$ and, by a Kuratowski's result\cite[Th\'{e}or\`{e}me A]{Kuratowski2}, the family  $LC^{n-1}$  is a one-to-one continuous image of a complete space, so it is Borel. Unfortunately, it is unknown if  the family   $\mathcal C_n\cap LC^{n-1}\cap C^{n-1}$ is Borel for any $n>1$ (for $n=1$ it is known to be $F_{\sigma\delta}$). We are able to show the Borel bireducibility of $\mathcal C_n$ to  $\mathcal C_n\cap LC^{n-1}$ continua. The strategy is first to Borel reduce $\mathcal C_n$ to $2^{M_n}$ (here, we do not know if this can be done continuously) and next, to proceed similarly as in the proof of Theorem~\ref{compAR}.

\begin{lemma}\label{Borel}
At most $n$-dimensional compacta in the Hilbert cube are Borel reducible to compacta in the universal $n$-dimensional Menger continuum for each $n\in\mathbb N$.
\end{lemma}

\begin{proof}
Let us denote by $H_p(Q, M_n)$ the set of all partial embeddings $f\colon K\to M_n$ of compact subsets $K\subseteq Q$.  It is shown in~\cite{Kuratowski1} to be a Polish space with the topology induced by the convergence
\[f_k\to f\quad \Leftrightarrow \quad d_H(dom f_k, dom f)\to 0\quad \text{and}\quad f_k(x_k)\to f(x)\]
for each sequence $x_k\to x$, $x_k\in dom f_k$, $x\in dom f$, where $dom $ denotes the domain of the respective mapping.
Consider the mapping  $\Phi\colon \mathcal C_n\to 2^{H_p(Q, M_n)}$ given by 
\[\Phi(K)=\{f\in H_p(Q, M_n)\colon dom f=K\}.\]
\begin{claim} $\Phi$ is lower semi-continuous.
\end{claim}
Indeed, let $U$ be an open subset of $H_p(Q, M_n)$. We are going to show that the set 
$$\{K\in \mathcal C_n\colon \Phi(K)\cap U\neq\emptyset\}$$
is open in $\mathcal C_n$. Suppose it is not. Then there is an embedding $f\colon K\to M_n$ in $U$ such that a sequence $K_i\in \mathcal C_n$ exists satisfying $K_i\to K$ and $\Phi(K_i)\cap U=\emptyset$. By the $Z$-approximation theorem for $M_n$~\cite[Theorem 2.3.8]{Bestvina}, we can assume that $f$ is a $Z$-embedding. Since $K\cup\bigcup_iK_i\in \mathcal C_n$ and  $ M_n$ is an absolute extensor for the class $\mathcal C_n$, there is a continuous  extension $\overline{f}\colon K\cup\bigcup_iK_i \to M_n$ of $f$. Again, by the $Z$-approximation theorem, $\overline{f}$ can be approximated by   $Z$-embeddings $g\colon K\cup\bigcup_iK_i \to M_n$ which agree with $f$ on $K$. It follows that mappings $g_i=g|K_i$ converge to $f$ in $H_p(Q, M_n)$, so $g_i\in U$ for sufficiently large $i$. This contradicts the equality $\Phi(K_i)\cap U=\emptyset$.

\

By applying the Kuratowski and Ryll-Nardzewski selection theorem~\cite{KRN} to the mapping $\Phi$ we get a Borel selection map $\varphi\colon \mathcal C_n\to H_p(Q, M_n)$. The required reduction is the function $\mathcal C_n \ni K\mapsto \pi\varphi(K) \in 2^{M_n}$, where $\pi$ is  the projection $Q\times M_n \to M_n$.
\end{proof}

\begin{theorem}\label{ndim}
At most $n$-dimensional compacta are Borel bireducible with at most $n$-dimensional $LC^{n-1}$-continua, for every $n\in\N$.
\end{theorem}

\begin{proof}
By Lemma~\ref{Borel}, it remains to continuously reduce the family $2^{M_n}$ to  the family of $LC^{n-1}$ continua in $Q$. Fix a pair $M_n\subseteq M'_n$ of topological copies of the universal $n$-dimensional Menger continua in $Q$ such that $M_n$ is a $Z$-set in $M'_n$ (clearly, this is possible by the $Z$-approximation theorem). Thus, for each $k\in \mathbb N$, there is a continuous mapping $f_k\colon M'_n\to M'_n\setminus M_n$ such that 
$$d_{sup}(f_k,id_{M'_n})< \frac1{k} \quad\text{and}\quad f_{k+1}(M_n)\cap f_{i}(M_n)=\emptyset \quad\text{for}\quad i<k+1.$$ Put $f_0=id_{M'_n}$.    There is also a homotopy $H\colon 2^{M_n}\to 2^{M_n}$ through finite sets by Lemma~\ref{homotopy}. Given $K\in 2^{M_n}$, let
\[\psi(K)=(K\times \{0\})\cup \bigcup_{k=0}^\infty f_k(H(K,2^{-k}))\times \{0\}\]
 and define a continuous mapping
 \[\varphi(K)=(M'_n\times  \{0\})\cup \bigcup_{k=0}^\infty f_k(H(K,2^{-k}))\times [0,2^{-k}].\]
 Hence, $\varphi(K)$ is the Menger continuum $M'_n$ with a null sequence of mutually disjoint segments attached, so it is $n$-dimensional. Moreover, observe that attaching to an $LC^{n-1}$-compactum finitely many disjoint arcs yields an $LC^{n-1}$-compactum and $\varphi(K)$ is $\varepsilon$-deformable into such a compactum for every $\varepsilon>0$. Thus, by Proposition~\ref{LC}, $\varphi(K)$ is an $LC^{n-1}$-continuum.

In order to show that $\varphi$ is a reduction, suppose $h\colon K\to L$ is a homeomorphism between $K,L\in 2^{M_n}$. Extend $\overline{h}(x,0):=(h(x),0)$, $x\in K$,  to a homeomorphism $h_1\colon \psi(K)\to \psi(L)$ by Proposition~\ref{extension}. Since  $\psi(K)$ and  $\psi(L)$ are $Z$-sets in  $M'_n \times  \{0\}$, there is a homeomorphism $h_2\colon M'_n \times  \{0\} \to M'_n \times  \{0\}$ that extends $h_1$~\cite[Corollary 3.1.5]{Bestvina}. Finally, extend $h_2$ to a homeomorphism $h_3\colon \varphi(K)\to \varphi(L)$ linearly over the attached segments.

Conversely, if $g\colon \varphi(K)\to \varphi(L)$ is a   homeomorphism then, as in the proof of Theorem~\ref{compAR}, we conclude that
\[g\left(\bigcup_{k=0}^\infty f_k(H(K,2^{-k}))\times \{0\}\right)= \bigcup_{k=0}^\infty f_k(H(L,2^{-k}))\times \{0\}\]
and, consequently, $g(K\times \{0\})=L\times \{0\}$,
hence $K$ and $L$ are homeomorphic.

\end{proof}

\begin{proposition}\label{uniformization}
Let $X$ be a standard Borel space, $Y$ be a Polish space and $A\subseteq X\times Y$ be a Borel set whose vertical sections are $\sigma$-compact. Then $\pi_X(A)$ is Borel and there exist Borel functions $f_n\colon \pi_X(A)\to Y$ whose graphs are subsets of $A$ and such that for every $x\in \pi_X(A)$ the set $\{f_n(x)\colon n\in\N\}$ is dense in the vertical section $A_x$.
\end{proposition}

\begin{proof}
Let $\{B_n\colon n\in\N\}$ be an open countable base of $Y$ with $B_1=Y$.
Notice that $A\cap (X\times B_n)$ has $\sigma$-compact vertical sections.
Hence, for every $n\in\N$ we can apply the Arsenin-Kunugui Uniformization Theorem \cite[Theorem 35.46]{Kechris} to the Borel set $A\cap (X\times B_n)\subseteq X\times Y$ to get a Borel uniformization $g_n\colon X_n\to B_n$ where $X_n=\pi_X(A\cap (X\times B_n))$ is Borel in $X$.
Now, it is enough to define
\[f_n(x)=\begin{cases}g_n(x), &x\in X_n,
\\ g_1(x), &x\in \pi_X(A)\setminus X_n. \end{cases}\]
Clearly all the functions $f_n$ are Borel and the set $\{f_n(x)\colon n\in\N\}$ is dense in $A_x$ for every $x\in X$.
\end{proof}

If $A\subseteq X$ is an arc with end points $a$ and $b$ such that $A\setminus \{a,b\}$ is an open subset of $X$ then $A\setminus \{a,b\}$ is called a\textit{ free arc} in $X$. If $A\setminus \{a,b\}$ is a free arc, then the union of all free arcs in $X$ containing $A\setminus \{a,b\}$ will be called a maximal free arc in $X$. 
 Let $A_X$ be the union of all free arcs in $X$. 
 
 Let us denote by $\mathcal R$ the set of all regular subcontinua of the Hilbert cube $Q$ and for $X\in\mathcal R$, denote by $B_X$ the set of all \emph{branch points} (i.e.,  points of order $\geq 3$) which are also \emph{local cut points} of $X$ (i.e., cut points of some open connected neighborhoods).

\begin{lemma}\label{separation}
Let $X$ be a regular continuum and let $x, y\in X$ be distinct. Then there is a finite set $H\subseteq A_X\cup B_X$ which separates $x$ from $y$. In particular,  $A_X\cup B_X$ is dense in $X$.
\end{lemma}

\begin{proof}
Let $F\subseteq X$ be a minimal finite set separating $x$ and $y$. Each point of $F$ is a local cut point of $X$~\cite[Corollary (9.42)]{Whyburn}.  Also, for every $s\in F$ there is an arc from $x$ to $y$  containing $s$ (otherwise, $F$ would not be a minimal separating set).  
Let $T\subseteq F$ be the set of  points with order two.
For every $t\in T$  choose an arc $A_t$ with end points  $x$ and $y$ containing $t$ and  an open connected neighborhood $U_t$ of $t$ whose closure is disjoint from $F\setminus\{t\}$ and such that $A_t\cap U_t$ is an arc without end points. We can assume that $t$ cuts $U_t$.

Let $\bd  (A_t\cap \overline{U_t})=\{a,b\}$  and 
\[M=\{p\in U_t\cap A_t\colon \text{$p$ separates $a$ and $b$ in $ \overline{U_t}$} \}.\]
Notice that  $t\in M$ since $F$ separates $x$ and $y$. 
It is enough to prove that at least one of the following holds:
\begin{itemize}[noitemsep]
\item[(1)] there is a branch point $z\in A_t\cap U_t$ which separates $a$ and $b$ in $ \overline{U_t}$, 
\item[(2)] there is a point $z\in A_t \cap U_t$ lying on a free arc which separates $a$ and $b$ in $ \overline{U_t}$.
\end{itemize}
Indeed, in both cases point $t$ can be replaced by $z$, i.e., the set $(F\setminus \{t\})\cup \{z\}$ separates $\{x\}$ and $\{y\}$; after finitely many such replacements we get desired set $H$.

So,  suppose  condition (1) is not satisfied. We prove that condition (2) holds true.
By our assumption every point of $M$ is a point of order two.
One can easily prove that 
\begin{claim}
\begin{itemize}[noitemsep]
\item[a)] Every point in $M$ is approximated by points in $M$ from both sides with respect to a natural order of $A_t\cap U_t$. Another way of stating this is that no point of $M$ is accessible from $(A_t\cap U_t)\setminus M$.
\item[b)] The set $M$ is closed in $A_t\cap U_t$.
\end{itemize}
\end{claim}
Proof of a)
If $z\in M$ then $z$ is of order two. We consider  two-point boundaries of small neighborhoods of $z$ contained in $U_t$. The boundary points  separate $a$ from $b$ in $\overline{U_t}$ and hence they are in $M$.

Proof of b) The limit of a sequence from $M$ is either outside $U_t$ or it separates $a$ from $b$ in $\overline{U_t}$. Hence the limit must be in $M$ as well.

Since $M\neq\emptyset$,  it follows by  a) and b)  that $M=A_t\cap U_t$.
Hence every point of $U_t$ is of order two and thus $U_t$ is a free arc and $t\in U_t$ separates $a$ and $b$ in $\overline{U_t}$.
Hence,  condition (2) is proved.
\end{proof}

\begin{definition}
Suppose that $X$ and $Y$ are regular continua and let $C\subseteq X$ and $D\subseteq Y$. A bijection $f\colon C\to D$ is called \emph{separation preserving} if, for every pair $x, y\in C$ and a finite set $F\subseteq C$, 
$x$ and $y$ are separated by $F$ if and only if $f(x)$ and $f(y)$ are separated by $f(F)$.
\end{definition}

\begin{lemma}\label{extensiondense}
Let $X$ and $Y$ be regular continua and let $C$ and $D$ be dense subsets of $X$ and $Y$ respectively, for which  $B_X\subseteq C\subseteq B_X\cup A_X$ and $B_Y\subseteq D\subseteq B_Y\cup A_Y$. Then every separation preserving bijection $f\colon C\to D$ can be extended to a homeomorphism of $X$ onto $Y$.
\end{lemma}

\begin{proof}
Since $X$ and $Y$ are compact there is only one uniformity on $X$ and $Y$ respectively. Let us prove that $f$ is uniformly continuous.
Suppose for contradiction that this is not the case. Then there exist  $\varepsilon>0$ and two sequences $(x_n), (y_n)$ such that the distances of $x_n$ and $y_n$ converge to zero, but the distance of $f(x_n)$ and $f(y_n)$ is bigger than $\varepsilon$. 
Since $X$ and $Y$ are compact, we can assume that $f(x_n)$ converges to some $a\in X$, $f(y_n)$ converges to some $b\in Y$ and both the sequence $x_n$ and $y_n$ converge to the same point $z\in X$. 


By Lemma~\ref{separation} there is a finite set $H_1\subseteq B_Y\cup A_Y$ separating $a$ and $b$. By a small adjustment of points in  $H_1\cap A_Y$ we can suppose that $H_1\subseteq D$.
Similarly, there is $H_2$ with the same properties which is disjoint from $H_1$. This follows, e.g.,  from~\cite[Proposition 10.18]{Nadler} where the ``additive-hereditary system'' is the family of finite subsets of $B_Y\cup A_Y$ and $H_2$ is the boundary of a neighborhood of $a$  disjoint from $H_1$, slightly adjusted to be in $D$. 
There is $i$ such that $f^{-1}(H_i)$ does not contain the point $z$.
Since $X$ is locally connected, there is a connected neighborhood $U$ of $z$ which is disjoint from $f^{-1}(H_i)$. For   $n$  sufficiently large, $x_n, y_n\in U$ and hence  $x_n$ and $y_n$ are not separated by $f^{-1}(H_i)$.

Since $Y$ is locally connected there are connected neighborhoods of $a$ and $b$ disjoint from $H_i$, hence it follows that for $n$  sufficiently large, $f(x_n)$ and $f(y_n)$ are separated by $H_i$. We get a contradiction, since the function  $f$ was supposed to be separation preserving.

Since $f$ is uniformly continuous it can be extended to a continuous mapping of $X$ to $Y$
\cite[Theorem~4.3.17]{Engelking}. Since, moreover, we can do the same with $f^{-1}$, the extended  mapping  is a homeomorphism.
\end{proof}

For the puprose of the following lemma and Theorem~\ref{rimfinitecontinua} let us fix the sets
\[\mathcal A=\{(X, a)\colon X\in\mathcal R, a\in A_X\},\quad
\mathcal B=\{(X, b)\colon X\in\mathcal R, b\in B_X\}.\]

\begin{lemma}\label{borel}
The sets  $\mathcal A$ and   $\mathcal B$ are Borel subsets of $C(Q)\times Q$.
\end{lemma}

\begin{proof}
We will mostly follow a corresponding idea  for dendrites from the proof of~\cite[Lemma 6.4]{CamerloDarjiMarcone}. So,  in order to show that the sets are Borel, we check that they are  both analytic and coanalytic. The proof for coanalyticity of $\mathcal B$ is exactly the same for regular continua as for dendrites. To show that $\mathcal B$ is analytic we have to modify the respective formula as follows:

$(X,b)\in \mathcal B$ iff $b\in X$, $X\in \mathcal R$ and there exist arcs $A_1, A_2, A_3\subseteq X$ such that $b\in A_1\cap A_2\cap A_3$ and for every $x\in Q$ and every distinct $i,j\in \{1,2,3\}$ if $x\in A_i\cap A_j$ then $x=b$.

Before considering  $\mathcal A$, let us first check that the set $\mathcal E:=\{(X,e)\colon X \in\mathcal R, e\in E_X\}$, where $E_X$ is the set of all points $e\in X$ of order $\ord_xX=1$, is Borel in  $C(Q)\times Q$.
Indeed, its analyticity can be shown as for the set $\mathcal E^{\mathcal D}$ in~\cite[Lemma 6.4]{CamerloDarjiMarcone}. To see that it is coanalytic, one can characterize $\mathcal E$ in the following way:

$(X,e)\in\mathcal E$ iff  $e\in X$, $X\in \mathcal R$ and,   for every two arcs $A_1,A_2\in C(Q)$, if $e\in A_1\cap A_2$ and $A_1,A_2\subseteq X$ then there is $x\in Q$ such that $x\in  A_1\cap A_2$ and $x\neq e$.

It follows that the set $\mathcal O:=\{(X,x)\colon X \in\mathcal R, \ord_xX=2\}=\{(X,x)\colon X \in\mathcal R, x\in X\}\setminus (\mathcal B \cup \mathcal E)$ is Borel.

Now, we are ready to characterize $\mathcal A$:

$(X,a)\in\mathcal A$ iff $(X,a)\in\mathcal O$ and  there exists $k\in \mathbb N$ such that, for each $x\in Q$, if $d(a,x)<\frac1{k}$ and $x\in X$ then $(X,x)\in \mathcal O$.

This shows coanalyticity of $\mathcal A$. On the other hand, the following description  gives its analyticity. Let $\{B_i\colon i\in\mathbb N\}$ be an open countable base in $Q$. Then

$(X,a)\in\mathcal A$ iff  $a\in X$, $X\in \mathcal R$ and there is an arc $A \in  C(Q)$ such that $a\in A$, $A\subseteq X$ and there is $i$ such that $a\in B_i$, $B_i\cap X \subseteq A$.
\end{proof}

\begin{theorem}\label{rimfinitecontinua}
Regular continua are classifiable by countable structures.
\end{theorem}

\begin{proof}
We will assign to every regular continuum $X$ a sequence $(R^X_n)_{n=1}^\infty$ of relations, such that $R^X_n\subseteq X^{n+2}$ for which $(x, y, x_1, x_2, \dots, x_n)\in R_n^X$ iff $\{x_1,\dots, x_n\}$ separates $x$ and $y$. 
Note that the sets $A_X$ are open in $X$ and hence $\sigma$-compact, and, by \cite[p. 606]{Whyburn}, the sets $B_X$ are countable for $X\in\mathcal R$. 
Notice that $A_X\cap B_X=\emptyset$ and that $A_X\cup B_X$ is dense in $X$ by Lemma~\ref{separation}.
Let $\mathcal R_A$, $\mathcal R_B$ denote the sets of continua $X\in\mathcal R$ such that $A_X\neq\emptyset$, $B_X\neq\emptyset$, respectively.

Clearly
\[\mathcal R= (\mathcal R\setminus \mathcal R_B)\cup (\mathcal R_B\setminus R_A)\cup (\mathcal R_B \cap R_A).\]
If we prove that each of the summands is Borel and it is classifiable by countable structures, it will follow that $\mathcal R$ is classifiable by countable structures as well.
The class $\mathcal R\setminus \mathcal R_B$ contains only three topological types of  continua: 
 a point, an arc and a simple closed curve. By \cite[Theorem~1]{Ryll-Nardzewski} it follows that they are Borel. Hence the set $\mathcal R\setminus \mathcal R_B$ is easily classifiable by countable structures (this is even a \emph{smooth} equivalence relation).
The case of $\mathcal R_B\setminus \mathcal R_A$ can be handled similarly as the remaining one, so we will skip it. In the rest we prove that $\mathcal R_B \cap R_A$ is classifiable by countable structures.

The vertical sections of the set $\mathcal A$ are $\sigma$-compact and $\mathcal A$ is Borel by Lemma~\ref{borel}.
Hence, by Proposition~\ref{uniformization} 
the set $\mathcal R_A$ is Borel and there are Borel mappings $b_{2n-1}\colon \mathcal R_A\to Q$ such that $\{b_{2n-1}(X)\colon n\in\N\}$ is dense in $A_X$. We can suppose that every point of the dense set occurs infinitely many times.

The vertical sections of $\mathcal B$ are countable and $\mathcal B$ is Borel by Lemma~\ref{borel}.
It follows by the Lusin-Novikov Uniformization Theorem \cite[Theorem 18.10]{Kechris} that $\mathcal R_B$ is Borel and there are countably many Borel mappings $b_{2n}\colon \mathcal R_B\to Q$ such that $\{b_{2n}(X)\colon n\in\N\}=B_X$. We can easily assume that every point occurs infinitely many times.

We define an $(n+2)$-ary relation $S^X_n$ on $\N$. For $k, l, m_1, \dots, m_n\in\N$ let
\[S^X_n(k, l, m_1, \dots, m_n) \quad\text{ iff }\quad
R^X_n\left(b_k(X), b_l(X), b_{m_1}(X), \dots, b_{m_n}(X)\right),\]
and let
\[\Phi(X)=\left(\N, (S^X_n)_{n\in\N}\right).\]
Roughly speaking, the mappings $b_n$ allow us to code  relations $R_n^X$ by relations on  $\N$ in a Borel way.

Let us check that $\Phi$ is a Borel reduction of $\mathcal R_A\cap \mathcal R_B$ to countable structures.
Suppose that $h\colon X\to Y$ is a homeomorphism and $X, Y\in\mathcal R_A\cap \mathcal R_B$. Clearly, $h(A_X)=A_Y$ and $h(B_X)=B_Y$. For any $x\in B_X$, the sets $\{n\in\N\colon b_{2n}(X)=x\}$ and
$\{n\in\N\colon b_{2n}(Y)=h(x)\}$ are infinite. So let $\varphi_x$ be an arbitrary bijection between them. 

Let $A$ be  a maximal  free arc in $X$.  Note that $\overline{A}$ is either an arc or a simple closed curve. The set $E_X=\{b_{2n-1}(X)\colon n\in\N\}\cap A$ is countable and dense in $A$. The same is true for the set $E_Y=\{b_{2n-1}(Y)\colon n\in\N\}\cap h(A)$ in $h(A)$. Since the real line is countable dense homogeneous, there is a homeomorphism $\psi_A\colon\overline{A}\to h(\overline{A})$ such that  $\psi_A|(\overline{A}\setminus  A)=h|(\overline{A}{}\setminus  A)$ and $\psi_A(E_X)=E_Y$.  Now, for every $x\in C$,  the sets $\{n\in\N\colon b_{2n-1}(X)=x\}$  and $\{n\in\N\colon b_{2n-1}=\psi_A(x)\}$ are infinite, hence  there is a bijection $\varphi_x$  between  these sets.
Let $\varphi=\bigcup\{\varphi_x\colon x\in b_\N(X)\}$.
It is straightforward to check that $\varphi\colon\Phi(X) \to \Phi(Y)$ is an isomorphism.

On the other hand,  assume that $\varphi\colon \Phi(X)\to \Phi(Y)$ is an isomorphism.
Let us define $h(b_k(X))=b_{\varphi(k)}(Y)$, which is a correct definition (not depending on $k$ but only on $b_k(X)$). Indeed if $b_{\varphi(k)}(Y)\neq b_{\varphi(l)}(Y)$, then using Lemma~\ref{separation} we can find a finite set $\{y_1, \dots, y_n\} \subseteq b_\N(Y)$ separating $b_{\varphi(k)}(Y)$ and  $b_{\varphi(l)}(Y)$. Suppose that $y_i=b_{m_i}(Y)$. Then $S_n^Y(\varphi(k), \varphi(l), m_1, \dots, m_n)$, thus  $S_n^X(k, l, \varphi^{-1}(m_1), \dots, \varphi^{-1}(m_n))$,  since $\varphi$ is an isomorphism.  Hence the points $b_k(X)$ and $b_l(X)$ are separated  and so they are  different points.  Similarly, if   $b_k(X)\neq b_l(X)$, then  $b_{\varphi(k)}(Y)\neq b_{\varphi(l)}(Y)$.  

Let $C=b_\N(X)$ and let $D=b_\N(Y)$. Clearly,  $h\colon C\to D$ is a separation preserving bijection and hence, by
Lemma~\ref{extensiondense},  it can be extended to a homeomorphism of $X$ onto $Y$.

\end{proof}

Let $c_0=\{x\in\R^\N\colon x_n\to 0\}$. For a set $X$ of sequences of real numbers we denote by $X/c_0$ the equivalence relation for which two sequences are equivalent if and only if their difference converges to zero.

\begin{lemma}\label{turb}
The equivalence relation $I^\N/c_0$ is not classifiable by countable structures.
\end{lemma}

\begin{proof}
By the proof of \cite[Lemma~6.2.2]{Kanovei} it follows that $\R^\N/c_0$ is Borel reducible to $I^\N/c_0$. The equivalence relation $\R^\N/c_0$ is known to be \emph{turbulent} and hence by \cite[Example~3.23]{Hjorth} it is not classifiable by countable structures.  Hence $I^\N/c_0$ is not classifiable by countable structures too.
\end{proof}

\begin{theorem}\label{rimfinitecompacta}
Rim-finite compacta are not classifiable by countable structures.
\end{theorem}

\begin{proof}
By Lemma~\ref{turb}, it is enough to prove that $I^\N/c_0$ is Borel reducible to rim-finite compacta.
Let $J=I\times \{0\}$. For $n\geq 3$ denote by $T_n\subseteq \R\times (-\infty, 0)\cup\{(0,0)\}$ a simple $n$-od with diameter less than $2^{-n-1}$ whose vertex is located at the point $(0, 0)$. For $a, b\in \R$ denote $T_n(a, b)= (a, b) + T_n$.
Let $\{q_n\colon n\in\N\}$ be a dense subset of $(0,1)$ enumerated without repetitions and let 
\[L=J\cup\bigcup T_{k_n}(q_n, 0)\]
where $k_n\in\N$ are odd,  $3 < k_1<k_2<\dots$ and $T_{k_n}(q_n, 0)$ is disjoint from $T_{k_m}(q_m, 0)$ for $m\neq n$.
For $x\in I^\N$ define a compact subset of the plane
\[\Phi(x)=L\cup\bigcup_{n\in\N} T_{2n+2}(x_n, 2^{-n}).\]

Let us verify that $\Phi$ is a reduction.
Suppose that $x, y\in I^\N$ and $x-y\in c_0$. 
There are homeomorphisms $h_n\colon T_{2n+2}(x_n, 2^{-n}) \to T_{2n+2}(y_n, 2^{-n})$, simply let $h_n(a, b)=(a-x_n+y_n, b)$.
It follows that we can define $h\colon \Phi(x)\to \Phi(y)$ in such a way that $h$ is the identity on $L$ and $h(z)=h_n(z)$ if $z\in T_{2n+2}(x_n, 2^{-n})$. One can easily check that $h$ is one-to one  continuous and thus it is a homeomorphism. 

On the other hand, suppose that $h$ is a homeomorphism of  $\Phi(x)$ and $\Phi(y)$ for $x, y\in I^\N$. In $\Phi(x)$ as well as in $\Phi(y)$,  there is at most one point of order $n$ for every $n\geq 4$. Hence $h$ needs to preserve these points, especially $h|J$ equals  the identity on a dense subset of $J$ and,  by continuity,  $h|J$ is the identity.
Moreover, we get $h((x_n, 2^{-n})=(y_n, 2^{-n})$. 
Since
\[|x_n-y_n|=\|(x_n,0)-(y_n, 0) \|\leq \|(x_n,0)-(y_n, 2^{-n})\|=\|h(x_n, 0)-h(x_n, 2^{-n})\|,\]
we get  $x-y\in c_0$, by the uniform continuity of $h$.

The reduction $\Phi$ is not only Borel, but even continuous. This easy verification is left to the reader.
\end{proof}

We should remark that in the preceding proof we used simple $n$-ods as countably many markers in order to distinguish some points. We can even prove that compacta of order at most two are not classifiable by countable structures. As in the preceding proof we find a Borel reduction of $I^\N/c_0$ to such compacta. For a sequence $(x_n)\in I^\N$ we attach to $I$ some compact countable spaces whose Cantor-Bendixson rank is some countable ordinal number instead of simple triods (we can use e.g. the ordinal $\omega^n+1$ with the order topology instead of the simple triod $T_n$). So we get a compact space $\Phi(x)$ with only one nondegenerate component all of whose points are of order at most two. Hence we get the following result.

\begin{theorem}\label{compactaordertwo}
Compacta of order at most two are not classifiable by countable structures.
\end{theorem}

\section{Auxiliary results and questions}
In this section we provide two simple ways how to prove that continua are Borel bireducible with universal orbit equivalence, once we know that compacta have this property. This simplifies the result of \cite{ChangGao} substantially. Also we prove that arc-like continua and hereditarily decomposable continua are not classifiable by countable structures.

\begin{fact}~\cite{Cook}\label{Cook}
There exists a nondegenerate continuum $C$ such that for every subcontinuum $D\subseteq C$ any continuous map $f\colon D\to C$ is either the identity or a constant map. Any continuum with these properties is called a Cook continuum. 
\end{fact}

A construction of the Cook continuum can be also found  in \cite[p. 319--341]{PultrTrnkova}.

\begin{note}\label{CompCont}
Below we present a simple construction of a continuous reduction of  compacta  to continua.

Let $C\subseteq Q$ be a Cook continuum. Fix a point $c\in C$.
Let $Q^\prime= Q\times C/ Q\times\{c\}$  be the quotient space of $Q\times C$ where $ Q\times\{c\}$ is shrunk to a point and denote by $q\colon Q\times C\to Q^\prime$ the corresponding quotient mapping. 
 To every $K\in 2^Q$ assign $\varphi(K)=q(K\times C)$.
Hence, $\varphi(K)$ can be viewed as a cone over $K$  with the Cook continuum instead of an arc  as its fiber.
Let us verify that $\varphi\colon 2^Q\to 2^{Q^\prime}$ is a reduction of the homeomorphism relation of compacta to the homeomorphism relation of continua.
It is obvious that if $K$ and $L$ are homeomorphic, then $\varphi(K)$ and $\varphi(L)$ are homeomorphic.
On the other hand,  assume that $h\colon\varphi(K)\to \varphi(L)$ is a homeomorphism. Let  $q(x, a)\in \varphi(K)$ and  $h(q(x, a))=q(y,b)\in q(L\times C)$. Suppose  that $a\neq b$.
Denote by $i_x\colon C\to q(\{x\}\times C)$ the embedding given by $i_x(d)=q(x, d)$.
Let $\tilde\pi_2\colon Q^\prime\to C$ be such that $\tilde \pi_2\circ q=\pi_2$, where $\pi_2$ is the projection of $ Q\times C$  onto $C$.
Consider the mapping $\tilde \pi_2\circ h\circ i_x\colon C\to C$ which maps the point $a$ to $b$. Since these points are distinct, the preceding mapping needs to be a constant map by the properties of the Cook continuum mentioned in Fact~\ref{Cook}.
It follows that $h$ maps the set $i_x(C)$ into the set $q(L\times\{b\})$. In particular, the point $q(x, c)$ is mapped to $q(x, b)$.
Since $q(x, c)=q(x^\prime, c)$ for every $x^\prime \in K$, it follows analogously that 
 $h$ maps every set $i_{x^\prime}(C)$ into the set $q(L\times\{b\})$. Thus,  $h$ is not an onto mapping, a contradiction. Consequently, $h$  maps $q(K\times\{a\})$ onto  $q(L\times\{a\})$ for every $a\in C$. Hence, $K$ and $L$ are homeomorphic.
It is also easy to see that $\varphi$ is continuous.
\end{note}

\begin{note}
There are more simple ways  of obtaining such a reduction. Let us sketch the one which  avoids a Cook continuum, but uses a nondegenerate arc-less continuum $B$.   Fix a point $b\in B$.  To a compactum $X$ assign its cone, $cone(X)=X\times I/X\times\{1\}$,  and attach $X\times B$ to  $X\times\{0\}\subseteq cone(X)$ along $X\times\{b\}$.  Thus obtained continuum $\varphi(X)$ has only one arc component with nonempty interior (namely the cone) and its boundary is $X\times\{0\}$. It follows easily that $\varphi$ is a continuous reduction of all compacta to continua.
\end{note}

\begin{proposition}
Neither arc-like continua nor hereditarily decomposable continua are  classifiable by countable structures.
\end{proposition}

\begin{proof}
Similarly to the proof of Theorem~\ref{rimfinitecompacta} it suffices to prove that $I^\N/c_0$ is Borel reducible to arc-like hereditarily decomposable continua.  We describe the idea just very informally.
For a sequence $(x_n)\in I^{\mathbb N}$,  consider a continuum in the plane consisting of an arc $I$ with a ray $L$  approximating $I$ from the left which makes the arc rigid for self-homeomorphisms. From the other side take another ray $R$ approximating $I$ with small two-sided $\sin(1/x)$- curves inserted into small neighborhoods of the points $A_1, A_2, \dots$ which  encode  sequence $(x_n)$. See the figure below.
\end{proof}

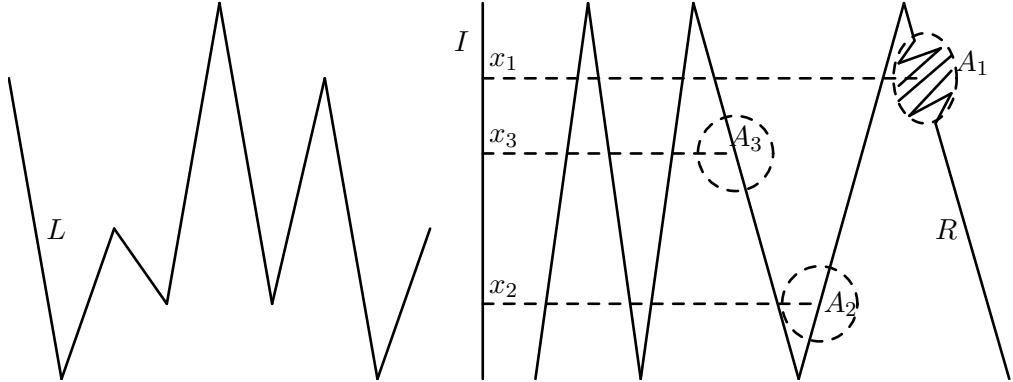
\begin{tikzpicture}[line cap=round,line join=round,>=triangle 45,x=0.7cm,y=1.0cm] 
\draw [color=black,, xstep=1.0cm,ystep=1.0cm] (0,0)  (0,0);
\clip(-9.0,0) rectangle (15.0,5.5);
\draw [line width=1.pt] (0.,0.)-- (0.,5.);
\draw [line width=1.pt] (-9.,4.)-- (-8.,0.);
\draw [line width=1.pt] (-8.,0.)-- (-7.,2.);
\draw [line width=1.pt] (-7.,2.)-- (-6.,1.);
\draw [line width=1.pt] (-6.,1.)-- (-5.,5.);
\draw [line width=1.pt] (-5.,5.)-- (-4.,1.);
\draw [line width=1.pt] (-4.,1.)-- (-3.,4.);
\draw [line width=1.pt] (-3.,4.)-- (-2.,0.);
\draw [line width=1.pt] (-2.,0.)-- (-1.,2.);
\draw [line width=1.pt] (8.,5.)-- (6.,0.);
\draw [line width=1.pt] (6.,0.)-- (4.,5.);
\draw [line width=1.pt] (4.,5.)-- (3.,0.);
\draw [line width=1.pt] (3.,0.)-- (2.,5.);
\draw [line width=1.pt] (2.,5.)-- (1.,0.);
\draw [line width=1.0pt,dash pattern=on 4pt off 4pt] (0.,4.)-- (8.4,4);
\draw [line width=1.0pt,dash pattern=on 4pt off 4pt] (0.,1.)-- (6.4,1);
\draw [line width=1.0pt,dash pattern=on 4pt off 4pt] (0.,3.)-- (4.8,3);
\draw [line width=1.0pt,dash pattern=on 4pt off 4pt] (8.4,4) circle (0.6);
\draw [line width=1.0pt,dash pattern=on 4pt off 4pt] (6.4,1) circle (0.5cm);
\draw [line width=1.0pt,dash pattern=on 4pt off 4pt] (4.8,3) circle (0.5cm);
\draw [line width=1.pt] (7.9,3.7)-- (8.9,4.3);
\draw [line width=1.pt] (8.,5.)-- (8.2,4.5);
\draw [line width=1.pt] (8.6,3.4)-- (10.,0.);
\draw [line width=1.pt] (8.6,3.4)-- (8.9,3.8);
\draw [line width=1.pt] (8.9,3.8)-- (8.1,3.5);
\draw [line width=1.pt] (8.1,3.5)-- (8.9,4.1);
\draw [line width=1.pt] (8.2,4.5)-- (7.9,4.2);
\draw [line width=1.pt] (7.9,4.2)-- (8.7,4.4);
\draw [line width=1.pt] (8.7,4.4)-- (7.9,3.9);
\draw[color=black] (-0.4,4.5) node {$I$};
\draw[color=black] (-8.1,2.0) node {$L$};
\draw[color=black] (0.4,4.22) node {$x_1$};
\draw[color=black] (9.3,4.2) node {$A_1$};
\draw[color=black] (0.4,1.2) node {$x_2$};
\draw[color=black] (6.8,1.0) node {$A_2$};
\draw[color=black] (0.4,3.2) node {$x_3$};
\draw[color=black] (5,3.2) node {$A_3$};
\draw[color=black] (8.8,2) node {$R$};
\end{tikzpicture}

\begin{note} Dendrites can be naturally coded either as a subset of the hyperspace of the Hilbert cube or as the hyperspace of all subcontinua of the universal dendrite. Fortunately both these codings are equivalent in the sense that the corresponding homeomorphism equivalence relations are Borel bireducible one with the other. This follows actually by the proof of \cite{CamerloDarjiMarcone}. Alternatively,  one can Borel reduce countable linear orders to dendrites in the universal dendrite. Since countable linear orders are Borel bireducible with the universal $S_\infty$ orbit equivalence relation we obtain the same for dendrites in the universal dendrite.
\end{note}

\begin{question}
 Are there some Borel reductions decreasing the topological dimension?
\end{question}

\begin{question}
Is the universal orbit equivalence Borel reducible to rim-finite compacta?
Is it reducible to compacta of order at most two?
\end{question}

\bibliographystyle{alpha}

\end{document}